\documentclass[a4paper]{amsart}
\usepackage{amssymb,amscd}
\usepackage[cp1251]{inputenc}
\usepackage[T2A]{fontenc}
\usepackage{graphicx}
\usepackage[all]{xy}

\input xypic

\newtheorem{theorem}{Theorem}[section]

\theoremstyle{definition}
\newtheorem{definition}[theorem]{Definition}
\newtheorem{example}[theorem]{Example}

\theoremstyle{remark}
\newtheorem*{remark}{Remark}

\numberwithin{equation}{section}

\renewcommand{\phi}{\varphi}
\newcommand{\bl}{\mbox{$\lambda\kern-0.53em\lambda$}}
\newcommand{\bmu}{\mbox{$\mu\kern-0.55em\mu$}}
\newcommand{\bnu}{\mbox{$\nu\kern-0.51em\nu$}}
\def\bphi{\mbox{$\varphi\kern-0.59em\varphi$}}

\def\R{\mathbb R}

\def\Z{\mathbb Z}

\def\sK{\mathcal K}

\newcommand{\mb}[1]{{\textbf {\textit#1}}}

\renewcommand{\ge}{\geqslant}
\renewcommand{\le}{\leqslant}

\newcommand{\obg}[1]{g_#1^{n_#1}}

\newcommand{\Ker}{\mathop{\rm Ker}}
\newcommand{\rank}{\mathop{\mathrm{rank}}}

\def\raag{\mbox{\it RA\/}}

\newcommand{\lk}{\mathcal L_{\mathcal K}}

\def\pt{\mathit{pt}}

\begin{document}

\title{On the commutator subgroup of a right-angled Artin group}

\author{Taras Panov}
\address{Department of Mathematics and Mechanics, Moscow
State University, Leninskie Gory, 119991 Moscow, Russia,
\newline\indent Institute for Theoretical and Experimental Physics,
Moscow, Russia \quad \emph{and}
\newline\indent Institute for Information Transmission Problems,
Russian Academy of Sciences} \email{tpanov@mech.math.msu.su}
\urladdr{http://higeom.math.msu.su/people/taras/}

\author{Yakov Veryovkin}
\address{
Steklov Mathematical Institute, Russian Academy of Sciences}
\email{verevkin\_j.a@mail.ru}

\thanks{The research of both authors was carried out at the Steklov Institute of Mathematics and
supported by the Russian Science Foundation grant no.~14-11-00414. The second author is a laureate of the Young Mathematics of Russia 2016 award.}

\subjclass[2010]{20F65, 20F12, 57M07}

\keywords{Right-angled Artin group, commutator subgroup, polyhedral product}

\begin{abstract}
We use polyhedral product models to analyse the structure of the commutator subgroup of a right-angled Artin group. In particular, we provide a minimal set of generators for the commutator subgroup, consisting of special iterated commutators of canonical generators.
\end{abstract}

\maketitle

\section{Introduction}
A \emph{right-angled Artin group}, also known as a \emph{graph group} or a \emph{partially commutative group}, has $m$ generators $g_1,\ldots,g_m$ and commutativity relations $g_ig_j=g_jg_i$ for some pairs~$\{i,j\}$. Such a group interpolates between a free group of rank $m$ (in which no pair of generators commutes), and a free abelian group of rank $m$ (in which every pair of generators commutes). A right-angled Artin group is therefore defined by a graph $\Gamma$ with $m$ vertices, where two vertices are joined by an edge whenever the two corresponding generators commute. Right-angled Artin groups are classical objects in geometric group theory.

Right-angled Artin groups are particular cases of \emph{graph
product groups}, corresponding to a sequence of $m$ groups $\mb
G=(G_1,\ldots,G_m)$ and a graph~$\Gamma$ on $m$ vertices.
The graph product group $\mb G^\Gamma$ consists of
words with letters from $G_1,\ldots,G_m$ in which the elements of
$G_i$ and $G_j$ with $i\ne j$ commute whenever $\{i,j\}$ is an
edge of~$\Gamma$. The graph product group $\mb G^\Gamma$
sits between the free product $G_1\star\cdots\star G_m$
(corresponding to a graph consisting of $m$ disjoint vertices) and
the cartesian product $G_1\times\cdots\times G_m$ (corresponding
to a complete graph). A right-angled Artin group is a graph product in which all $G_i$ are infinite cyclic groups~$\Z$.

To each graph $\Gamma$ without loops and double edges one can assign its \emph{clique complex} $\sK$, which is the simplicial complex on the vertex set of $\Gamma$ whose simplices are cliques (complete subgraphs) of~$\Gamma$. The simplicial complex $\sK$ is \emph{flag}, and its one-skeleton is the graph~$\Gamma$. Many properties of graph products are formulated more explicitly in terms of the simplicial complex~$\sK$. We use the notation $\mb G^\sK$ for a graph product, alongside with~$\mb G^\Gamma$.

A graph product group $\mb G^\sK$  can be defined as the colimit of a diagram of groups over the face category of~$\sK$~\cite{pa-ve16}. There is a similar colimit construction of topological spaces, known as the \emph{polyhedral product}, which is fundamental in toric topology. The polyhedral product functor assigns a space $(\mb X,\mb A)^\sK$ to a sequence of $m$ pairs of spaces $(\mb X,\mb A)=\{(X_1,A_1),\ldots,(X_m,A_m)\}$ and a simplicial complex $\sK$ on $m$ vertices~\cite{bu-pa00,b-b-c-g10,bu-pa15}. From the homotopy-theoretical point of view, polyhedral products $(\mb X,\mb A)^\sK$ have much richer structure than graph product groups $\mb G^\sK$, as the former may be used to study all sorts of higher homotopy phenomena. However, when the underlying simplicial comlex $\sK$ is flag, the classifying space functor $B$ takes the graph product $(\mb G)^\sK$ to the polyhedral product~$(B\mb G)^\sK$, and the polyhedral product $(E\mb G,\mb G)^\sK$ is the classifying space for the commutator subgroup~$(\mb G^\sK){\vphantom{\bigr)}}'$ when all $G_i$ are abelian.
In the case of a right-angled Artin group, each classifying space $BG_i=B\Z$ is a circle, and we obtain as $(B\mb G)^\sK$ the subcomplex $(S^1)^\sK$ in the $m$-torus, introduced by Kim and Roush in~\cite{ki-ro80}.

In this paper we use the construction of polyhedral products and other techniques of toric topology to
study the commutator subgroups of right-angled Artin groups and more general graph products.
Apart from a purely algebraic interest, our motivation lies in the fact that these commutators subgroups arise as the fundamental groups of some remarkable aspherical spaces. We refer to~\cite{pa-ve16,b-e-m-p-p17} for the discussion of topological aspects of the theory of right-angled Artin groups and their commutator subgroups.

In Section~\ref{secprelim} we include some basic information about right-angled Artin groups and polyhedral products. In Section~\ref{secbasisfree} we construct a special basis for the commutator subgroup of a free group, which is also a right-angled Artin group corresponding to $m$ disjoint vertices.
Several other constructions of bases in the commutator subgroup of a free group are known, see e.\,g.~\cite[Appendix~1]{bo-ku94} and~\cite{hu-wa96}.
In Section~\ref{secbasisart} we give an explicit minimal set of generators for the commutator subgroup of an arbitrary right-angled Artin group~$\raag_\sK$. The generators are nested iterated commutators of powers of the canonical generators of~$\raag_\sK$ appearing in a special order prescribed by the combinatorics of~$\sK$. In Section~\ref{secgp} we generalise the result to arbitrary graph products.

We thank the anonymous referee for the most helpful comments and suggestions.

\section{Preliminaries}\label{secprelim}
Let $\sK$ be an abstract simplicial complex on the set $[m] = \{1,2, \dots, m \}$. We refer to a subset $I=\{i_1,\ldots,i_k\}\in\mathcal K $ as a \emph{simplex} (or \emph{face}) of~$\sK$. We assume that $\sK$ contains the empty set $\varnothing$ and all one-element subsets $\{i\}\in[m]$; the latter are referred to as the \emph{vertices} of~$\sK$. 

We denote by $F_m$ or $F(g_1, \ldots, g_m)$ a free group of rank $m$ with generators $g_1, \ldots, g_m$.

\begin{definition}
The \emph{right-angled Artin group} $\raag_\sK$ corresponding to~$\sK$ is defined by generators and relations as follows:
\[
  \raag_\sK = F(g_1,\ldots,g_m)\big/ (g_ig_j=g_jg_i\text{ for
  }\{i,j\}\in\sK).
\]
\end{definition}

Clearly, the group $\raag_\sK$ depends only on the $1$-skeleton of~$\sK$, the graph~$\sK^1$.

\begin{definition}\label{polpr}
Let
\[
  (\mb X,\mb A)=\{(X_1,A_1),\ldots,(X_m,A_m)\}
\]
be a sequence of $m$ pairs of pointed topological spaces, $\pt\in
A_i\subset X_i$, where $\pt$ denotes the basepoint. For each
subset $I\subset[m]$ we set
\begin{equation}\label{XAI}
  (\mb X,\mb A)^I=\bigl\{(x_1,\ldots,x_m)\in
  \prod_{k=1}^m X_k\colon\; x_k\in A_k\quad\text{for }k\notin I\bigl\}
\end{equation}
For a simplicial complex $\sK$ on $[m]$, the \emph{polyhedral product} of $(\mb X,\mb A)$
corresponding to $\sK$ is
\[
  (\mb X,\mb A)^{\sK}=\bigcup_{I\in\mathcal K}(\mb X,\mb A)^I=
  \bigcup_{I\in\mathcal K}
  \Bigl(\prod_{i\in I}X_i\times\prod_{i\notin I}A_i\Bigl),
\]
where the union is taken inside the Cartesian product $\prod_{k=1}^m X_k$.

In the case when all pairs $(X_i,A_i)$ are the same, i.\,e.
$X_i=X$ and $A_i=A$ for $i=1,\ldots,m$, we use the notation
$(X,A)^\sK$ for $(\mb X,\mb A)^\sK$. Also, if each $A_i=\pt$, then
we use the abbreviated notation $\mb X^\sK$ for $(\mb X,\pt)^\sK$,
and $X^\sK$ for $(X,\pt)^\sK$.
\end{definition}

\begin{example}\label{ppexa}\

\textbf{1.} Let $(X,A)=(S^1,\pt)$, where $S^1$ is a circle. The polyhedral product
\[
  (S^1)^\sK=\bigcup_{I\in\sK}(S^1)^I
\]
is a subcomplex of the $m$-torus~$(S^1)^m$, sitting between the $m$-fold wedge $(S^1)^{\vee m}$ and the $m$-fold Cartesian product~$(S^1)^m$.

\smallskip

\textbf{2.} Let $(X,A)=(\R,\Z)$, where $\Z$ is the set of integer points on the real line~$\R$. We denote the corresponding polyhedral product by~$\lk$:
\begin{equation}\label{lk}
  \lk=(\R,\Z)^\sK=\bigcup_{I\in\sK}(\R,\Z)^I\subset \R^m.
\end{equation}
When $\sK$ consists of $m$ disjoint points, $\lk$ is a grid in
$m$-dimensional space $\R^m$ consisting of all lines parallel to
one of the coordinate axis and passing though integer points. When
$\sK=\partial\varDelta^{m-1}$, the complex $\lk$ is the union of
all integer hyperplanes parallel to the coordinate hyperplanes.
\end{example}

A \emph{missing face} (or a \emph{minimal non-face}) of $\sK$ is a
subset $I\subset[m]$ such that $I$ is not a simplex of~$\sK$, but
every proper subset of $I$ is a simplex of~$\sK$. A simplicial
complex $\sK$ is called a \emph{flag complex} if each of its
missing faces consists of two vertices. Equivalently, $\sK$ is
flag if any set of vertices of $\sK$ which are pairwise connected
by edges is a simplex.

A \emph{clique} (or a \emph{complete subgraph}) of a graph
$\Gamma$ is a subset $I$ of vertices such that every two vertices
in $I$ are connected by an edge. Each flag complex $\sK$ is the
\emph{clique complex} of its one-skeleton $\Gamma=\sK^1$, that is,
the simplicial complex formed by filling in each clique of
$\Gamma$ by a face.

A path-connected space $X$ is \emph{aspherical} if
$\pi_i(X)=0$ for $i\ge2$. An aspherical space $X$ is an
Eilenberg--MacLane space $K(\pi,1)$ with $\pi=\pi_1(X)$, and also the \emph{classifying space} $B\pi$
for the group~$\pi$.

By the result of Kim and Roush~\cite[Theorem~10]{ki-ro80}, the polyhedral product  $(S^1)^\sK$ corresponding 
to a flag complex~$\sK$ is the classifying space for the right-angled Artin group~$\raag_\sK$.
The polyhedral product $\lk$ is the classifying space for the commutator subgroup~$\raag'_\sK$.
More precisely, we have the following result.

\begin{theorem}[see~{\cite[Corollary~3.3]{pa-ve16}}]\label{artfund}
Let $\sK$ be a simplicial complex on $m$ vertices, and let
$\raag_\sK$ be the corresponding right-angled Artin group.
\begin{itemize}
\item[(a)] $\pi_1((S^1)^\sK)\cong\raag_\sK$.
\item[(b)] $\pi_1(\lk)\cong\raag'_\sK$.
\item[(c)] Each of the spaces $(S^1)^\sK$ and $\lk$ is aspherical if and only if 
$\sK$ is a flag complex.
\end{itemize}
\end{theorem}

Given a subset~$J\subset[m]$, consider the restriction of $\sK$
to~$J$:
\[
  \sK_J=\{I\in\sK\colon I\subset J\},
\]
which is also known as a \emph{full subcomplex} of~$\sK$.

Recall that the \emph{smash product} of two pointed space $X$ and $Y$ is defined as $X\wedge Y=X\times
Y/(X\times\pt\cup\pt\times Y)$, that is, the quotient of the product $X\times Y$ by 
the wedge $X\vee Y$. Given a sequence of spaces $\mb A=(A_1,\ldots,A_m)$ and a subset
$J=\{j_1,\ldots,j_k\}\subset[m]$, we denote the smash product 
$A_{j_1}\wedge\cdots\wedge A_{j_k}$ by $\mb A^{\wedge J}$.

A theorem of Bahri, Bendersky, Cohen and Gitler~\cite[Theorem~2.21]{b-b-c-g10} describes the homotopy type of the suspension of a polyhedral product $(\mb X,\mb A)^\sK$ in the case when all $X_i$ are contractible. It implies the following result.

\begin{theorem}[{\cite{b-b-c-g10}}]\label{wedgethmX}
Let $\sK$ be a simplicial complex on~$[m]$, and let 
$(\mb X,\mb A)$ be a sequence of pairs of pointed cell complexes, in which all 
$X_i$ are contractible. Then there is the following isomorphism of integer homology groups:
\[
  H_p\bigl((\mb X,\mb A)^\sK\bigr)\cong
  \widetilde H_{p-1}\Bigl(\bigvee_{J\subset[m]}|\sK_J|\wedge\mb A^{\wedge J}\Bigr).
\]
\end{theorem}

\section{A basis for the commutator subgroup of a free group $F_m$}\label{secbasisfree}
When $\sK$ consists of $m$ disjoint points, the right-angled Artin group $\raag_\sK$ is a free group $F_m = F(g_1, \ldots, g_m)$. Here we construct a special basis for the commutator subgroup~$F'_m$.

Let $(g,h)=g^{-1}h^{-1}gh$ denote the group commutator of two elements $g,h$.

\begin{theorem}\label{basisraagfree}
The commutator subgroup $F'_m$ of a free group $F_m = F(g_1, \ldots, g_m)$ is  a free group
freely generated by nested iterated commutators of the form
\begin{equation}\label{itercomm}
  (\obg{j}, \obg{i}),
  (\obg{{k_1}},(\obg{j},\obg{i})),\ldots,(\obg{{k_1}},(\obg{{k_2}},
  \ldots,(\obg{{k_{m-2}}},(\obg{j},\obg{i}))\ldots)),
\end{equation}
where $n_k \in \mathbb{Z}\setminus\{0\}$, $k_1 < k_2 < \cdots < k_{m-2} < j > i$ and $k_s \neq i$ for any $s$.
\end{theorem}

\begin{proof}
Consider the covering of the wedge of circles $B = (S^1)^{\vee m}$ corresponding to the commutator subgroup of $\pi_1(B)=F_m$. The total space of this covering can be described as $E = E_m=
\bigcup_{i=1}^m L_i \subset \R^m$, were $L_i$ is the union of all lines in~$\R^m$ which are parallel to the $i$th coordinate axis and pass through integer points. We have $E_m=(\R,\Z)^\sK$, where $\sK$ is a simplicial complex consisting of $m$ disjoint vertices, see~\eqref{lk}. The covering map $E\to B$ takes all points from the integer lattice $\Z^m\subset E$ to the basepoint of the wedge, and takes all segments of the $i$th coordinate direction in the graph $E$ to the $i$th circle of the wedge~$B$. The induced homomorphism $\pi_1(E)
\rightarrow \pi_1(B)$ is identified with the inclusion of the commutator subgroup $F_m'
\hookrightarrow F_m$.

\smallskip

\emph{Case $m = 2$.} Let $F_2 = F(g_1, g_2)$. We need to check that the elements
$(\obg{2}, \obg{1})$ with $n_k \in \mathbb{Z}\setminus\{0\}$ form a basis of~$F'_2$.
A word in $F_2$ belongs to $F_2'$ if and only if the sum of the exponents of all $g_1$ is zero and the sum of the exponents of all $g_2$ is zero. Consider a reduced word $\phi$ from~$F_2'$. We need to write $\phi$ as a product of commutators $(\obg{2}, \obg{1})$. To do this, we shall split off commutators of this form one by one from the left. Denote by $|\phi|$ the number of alterations of $g_1$ and $g_2$ in the word~$\phi$. Note that if $|\phi| \le 2$, then $\phi = 1$, as $\phi \in F_2'$. We induct on the number~$|\phi|$. Assume that $\phi = g_1^i g_2^j g_1 \psi$ (the case $\phi = g_2^i g_1^j g_2 \psi$ is considered similarly). We rewrite the word as follows: $\phi = g_1^i g_2^j g_1^{-i} g_2^{-j} g_2^j g_1^i g_1\psi = (g_2^{-j}, g_1^{-i}){\vphantom{A}}^{-1}g_2^j g_1^{i+1}
\psi$. Now $\varphi$ is the product of a commutator of the required form and the word $\phi_1 = g_2^j g_1^{i+1} \psi$, where $|\phi_1| <|\phi|$. Therefore, after finitely many iterations we arrive at a word
$\phi_k$ with $|\phi_k| \le 2$, i.\,e. $\phi_k = 1$.

\smallskip

\emph{Case $m = 3$.} Let $F_3 = F(g_1, g_2, g_3)$. We need to check that the elements
\begin{equation}\label{itercomm3}
  (\obg{3},\obg{1}),\quad(\obg{3},\obg{2}),\quad(\obg{2},\obg{1}),
  \quad(\obg{1},(\obg{3},\obg{2})),\quad(\obg{2},(\obg{3},\obg{1}))
\end{equation}
with $n_k \in \mathbb{Z}\setminus\{0\}$ form a basis of~$F'_3$.

Let $(y_1, y_2, y_3)$ be Cartesian coordinates in $\mathbb{R}^3$. We contract the maximal tree in $E$ given by the union of all lines of the form
\[
  \{y_1 = C_1,\, y_2 = C_2\}, \quad \{y_2 = C_2,\, y_3 = 0\}, \quad
  \{y_1 = y_3 = 0\},
\]
where $C_1,C_2 \in \mathbb{Z}$ are arbitrary integer constants. Denote by~$\widehat{E}$ the space obtained as the result of contraction. Then $\widehat{E}$ is a wedge of circles, and the circles in this wedge correspond bijectively to those edges of the graph $E$ that do not belong to the contracted maximal tree. These edges of $E$ are segments of the form
\[
  [(c_1, c_2,
  c_3), \; (c_1 + 1, c_2, c_3)] \text{ \quad or \quad } [(c_1, c_2,
  c_3), \; (c_1, c_2 + 1, c_3)],
\]
where $c_i \in \mathbb{Z}$  (we also have $c_3\ne0$ in the former case, but this does not affect the subsequent argument). Every such segment can be completed to a loop in $E$ by adding edges from  the contracted maximal tree, and we need to express the word corresponding to this loop  via the commutators~\eqref{itercomm3}.

We first consider a segment $[(c_1, c_2, c_3), \; (c_1 + 1, c_2,c_3)]$. 
It corresponds to a circle in the wedge $\widehat{E}$ and to a loop in~$E$, all of whose segments except
$[(c_1, c_2, c_3), \;{(c_1 + 1, c_2,c_3)}]$ lying in the contracted maximal tree. This loop is given by the word $g_2^{c_2} g_1^{c_1} g_3^{c_3} g_1 g_3^{-c_3}g_1^{-c_1-1} g_2^{-c_2}$, and can be expressed via the commutators~\eqref{itercomm3} as follows:
\begin{multline*}
g_2^{c_2} g_1^{c_1} g_3^{c_3} g_1 g_3^{-c_3} g_1^{-c_1-1}
g_2^{-c_2} \\
= (g_2^{-c_2},(g_3^{-c_3},g_1^{-c_1})) \; (g_3^{-c_3},g_1^{-c_1})^{-1} \; 
(g_3^{-c_3}, g_1^{-c_1-1}) \; (g_2^{-c_2}, (g_3^{-c_3}, g_1^{-c_1 -1}))^{-1}.
\end{multline*}

Now consider a segment $[(c_1, c_2, c_3), \; (c_1, c_2 + 1,c_3)]$. 
It corresponds to the loop in $E$ given by the word 
$g_2^{c_2} g_1^{c_1} g_3^{c_3} g_2 g_3^{-c_3} g_1^{-c_1} g_2^{-c_2 - 1}$,
and can be expressed via the commutators~\eqref{itercomm3} as follows:
\begin{multline*}
g_2^{c_2} g_1^{c_1} g_3^{c_3} g_2 g_3^{-c_3} g_1^{-c_1} g_2^{-c_2-1}
\\= (g_2^{-c_2}, g_1^{-c_1}) \; (g_1^{-c_1},(g_3^{-c_3}, g_2^{-c_2})) \;
(g_3^{-c_3},g_2^{-c_2})^{-1} 
\\ (g_3^{-c_3}, g_2^{-c_2 - 1}) \;
\; (g_1^{-c_1},(g_3^{-c_3}, g_2^{-c_2 - 1}))^{-1} \; (g_2^{-c_2 - 1},g_1^{-c_1})^{-1}.
\end{multline*}

\smallskip

\emph{General case.} We start by describing a maximal tree in $E=E_m$. We construct this tree inductively as a union of lines. For the two-dimensional case, denote the coordinates in $\R^2$ by $(y_1, y_m)$. We take as a maximal tree in $E=E_2$ the union of the lines
$\{y_1 = C_1\}$ with $C_1 \in \mathbb{Z}$ and the line $\{y_m = 0\}$. Now assume that a maximal tree is chosen in $E_{k+1}\subset \R^{k+1}$ and describe how to chose such a tree in
$E_{k+2} \subset \R^{k+2}$. Denote the coordinates in $\R^{k+1}$ by $(y_1, \ldots, y_k, y_m)$, $k < m - 1$. The maximal tree in $E_{k+1}$ constructed on the previous step is a union of $(k+1)$ families of lines $\{Q_1\}, \ldots, \{Q_{k+1}\}$. Then we pass to $\R^{k+2}$ with coordinates $(y_1, \ldots, y_{k+1},y_m)$ and take  as a maximal tree in $E_{k+2}$ the union of the following families of lines:
\begin{multline*}
  \{Q_1;\; y_{k+1} = C_{k+1}\}, \quad
  \{Q_2;\; y_{k+1} = C_{k+1}\}, \ldots,\\
  \{Q_{k+1};\; y_{k+1} = C_{k+1}\}, \quad \{y_1=\cdots=y_k=y_m =
  0\}.
\end{multline*}
Note that we appended each family from the previous step by an additional equation
$y_{k+1} = C_{k+1}$, $C_{k+1} \in \mathbb{Z}$, and added one more family consisting of a single line $\{y_1=\cdots=y_k=y_m = 0\}$. At the end we obtain in $E = E_m$ a maximal tree which a union of $m$ families of lines. After contracting this maximal tree we obtain a wedge of circles $\widehat{E}$. The circles in this wedge correspond bijectively to those edges of the graph $E$ that do not belong to the contracted maximal tree. These edges of $E$ are segments of the form
\begin{equation}\label{Isegment}
  I = [(c_1, \ldots,
  c_m), \; (c_1, \ldots, c_p + 1, \ldots, c_m)], \quad c_i\in\Z,\;p \in \{1, \ldots, m - 1\}.
\end{equation}
where $c_i \in \mathbb{Z}$. Every such segment can be completed to a loop in $E$ by adding edges from  the contracted maximal tree, and we need to express the word corresponding to this loop  via the commutators~\eqref{itercomm}.

First consider a segment~\eqref{Isegment} with $p \neq m-1$. This segment $I$ corresponds to a loop in $E$ consisting of $I$ and segments from the contracted maximal tree, given by the word
$$
\chi = g_{m-1}^{c_{m-1}} g_1^{c_1} g_2^{c_2} \cdots
g_{m-2}^{c_{m-2}} g_m^{c_m} g_p g_m^{-c_m} g_{m-2}^{-c_{m-2}}
\cdots g_p^{-c_p-1} \cdots g_1^{-c_1} g_{m-1}^{-c_{m-1}}.
$$
Let $\widehat{I}$ be the orthogonal projection of $I$ onto the hyperplane $y_{m-1} = 0$. The loop corresponding to $\widehat{I}$ can be expressed via the commutators~\eqref{itercomm} not containing $g_{m-1}$, by the induction assumption. Let this expression be given by a word~$\psi$. Then $\chi =
(g_{m-1}^{-c_{m-1}}, \psi) \cdot \psi^{-1}$. The commutator $(g_{m-1}^{-c_{m-1}}, \psi)$ can also be expressed via the commutators~\eqref{itercomm}. To do this we use the identity
\begin{multline}\label{swap}
  (\obg{q},(\obg{p},x))=(\obg{q},x)(x,(\obg{p},\obg{q}))(\obg{q},\obg{p})(x,\obg{p})(\obg{p},(\obg{q},x))\\
  (x,\obg{q})(\obg{p},\obg{q})(\obg{p},x).
\end{multline}
This identity allows us to swap the element $g_{m-1}^{-c_{m-1}}$ with other elements~$g_i^{-c_i}$ in iterated commutators until $g_{m-1}^{-c_{m-1}}$ reaches the position prescribed by the ordering of indices in~\eqref{itercomm}. This argument is the same as in~\cite[Lemma~4.7]{pa-ve16}, and more details can be found there.

Now consider a segment~\eqref{Isegment} with $p = m-1$, that is,
\[
  I =  [(c_1, \ldots, c_m), \; (c_1, \ldots, c_{m-1} + 1, c_m)].
\]
The corresponding loop in $E$ is given by the word
\begin{multline*}
\chi = g_{m-1}^{c_{m-1}} g_1^{c_1} g_2^{c_2} \ldots
g_{m-2}^{c_{m-2}} g_m^{c_m} g_{m-1} g_m^{-c_m} g_{m-2}^{-c_{m-2}}
\ldots g_1^{-c_1} g_{m-1}^{-c_{m-1}-1} =\\= g_{m-1}^{c_{m-1}} \psi
g_{m-1} \xi g_{m-1}^{-c_{m-1} - 1}=(g_{m-1}^{-c_{m-1}},
\psi)\,\psi^{-1}\, (g_{m-1}^{-c_{m-1} - 1}, \xi)\, \xi^{-1},
\end{multline*}
where $\psi$ and $\xi$ do not contain $g_{m-1}$. The words $\psi$ and $\xi$ can be expressed via the commutators~\eqref{itercomm} not containing $g_{m-1}$, by the induction assumption. The rest of the argument is the same as in the previous paragraph.

\smallskip

We have therefore proved that any element of the commutator subgroup $F'_m$ can be expressed via the iterated commutators~\eqref{itercomm}. It remains to prove that the system of generators~\eqref{itercomm} is free. Let $I^m(s)$ be a cube in $\R^m$  with edges of length $s\in\Z$ and vertices in the lattice $\Z^m$. Let $E_m^{(s)}:=E_m\cap I^m(s)$. Then
$E_m^{(s)}$ is a finite graph, and $G^{(s)}_m=\pi_1(E_m^{(s)})$ is a free group of finite rank. We can position the cubes $I^m(s)$, $s=1,2,\ldots$, in such a way that they include in each other and $\R^m=\bigcup_{s=1}^\infty I^m(s)$. The free group
$F'_m=\pi_1(E_m)$ of infinite rank is the inductive limit of the groups~$G^{(s)}_m$. It is enough to prove that each group~$G^{(s)}_m$ is freely generated by the iterated commutators~\eqref{itercomm} corresponding to loops inside the cube~$I^m(s)$. For simplicity, assume that $I^m(s)$ is the cube in the positive orthant of $\R^m$ with a vertex at the origin. Then a commutator of the form~\eqref{itercomm} corresponds to a loop in~$I^m(s)$ if and only if all exponents of the elements in the commutator satisfy $0<n_k\le s$. The number of commutators~\eqref{itercomm} with exponents in this range is given by
\begin{equation}\label{Jm}
  J_m^{(s)}=\sum_{i=2}^m \binom m i (i-1)s^i.
\end{equation}
Denote by $W_m^{(s)}$ the rank of the free group $G^{(s)}_m$. It is equal to the number of circles in the wedge obtained by contracting a maximal tree in the graph~$E_m^{(s)}$. We need to show that $J_m^{(s)} = W_m^{(s)}$.

By considering a maximal tree in the graph $E_2^{(s)}$ we obtain $W_2^{(s)} = s^2$. Furthermore, we have a recursive relation
\[
W_m^{(s)} = W_{m-1}^{(s)}(s+1) + (s+1)^{m-1} s - s.
\]
It implies that
\[
  W_m^{(s)} = W_{m-l}^{(s)} (s+1)^l + l
  (s+1)^{m-1} s - s \sum_{i=0}^{l-1}(s+1)^i,
\]
for $0 \le l \le m-2$. In particular, for $l = m-2$ we obtain
\begin{equation}\label{Wm}
  W_m^{(s)} = s^2 (s+1)^{m-2} + (m-2) (s+1)^{m-1} s - s
  \sum_{i=0}^{m-3}(s+1)^i.
\end{equation}
Since both $J_m^{(s)}$ and $W_m^{(s)}$ are polynomials in $s$, we only need to compare their coefficients. Calculating the coefficient of $s^k$ in~\eqref{Wm} we obtain
\begin{multline*}
\binom{m-2}{k-2} + (m-2)\binom{m-1}{k-1} - \sum_{i=k-1}^{m-3}
\binom{i}{k-1} \\ = \binom{m-2}{k-2} + (m-2)\binom{m-1}{k-1} -
\binom{m-2}{k} 
\\  = m\binom{m-1}{k-1} - \binom{m}{k} = (k-1)\binom{m}{k}.
\end{multline*}
This is precisely the coefficient of $s^k$ in $J_m^{(s)}$, see~\eqref{Jm}.
\end{proof}

\section{Generators for the commutator subgroup of a right-angled Artin group}\label{secbasisart}

\begin{theorem}\label{raagbasis}
Let $\raag_\sK$ be the right-angled Artin group corresponding to a simplicial complex~$\sK$ on $m$ vertices. The commutator subgroup $\raag'_\sK$ has a minimal generator set consisting of iterated commutators
\begin{equation}\label{itercommu}
  (\obg{j}, \obg{i}),\;
  (\obg{{k_1}},(\obg{j},\obg{i})),\;\ldots,\;
  (\obg{{k_1}},(\obg{{k_2}},\ldots,(\obg{{k_{m-2}}},(\obg{j},\obg{i}))\ldots)),
\end{equation}
where $n_k \in \mathbb{Z}\setminus\{0\}$, $k_1 < k_2 < \cdots < k_{m-2} < j > i$, $k_s \neq i$ for any $s$, and $i$ is the smallest vertex in a connected component  not containing~$j$ of the subcomplex $\sK_{\{k_1,\ldots,k_{\ell-2},j,i\}}$.
\end{theorem}

\begin{proof}
In the case when $\sK$ consists of $m$ disjoint points the statement is proved in Lemma~\ref{basisraagfree}. Adding an edge $\{p,q\}$ to the complex $\sK$ results in adding the commutation relation $(g_p,g_q)=1$ to the right-angled Artin group. We shall eliminate commutators from the set~\eqref{itercomm} which do not appear in~\eqref{itercommu} using the new commutation relations. This argument is similar to the corresponding argument in the case of a right-angled Coxeter group, see~\cite[Theorem~4.6]{pa-ve16}, and we only outline the main steps.

First assume that the vertices $j$ and $i$ are in the same connected component of the complex
$\sK_{\{k_1,\ldots,k_{\ell-2},j,i\}}$.  We shall
show that the corresponding commutator $(\obg{{k_1}},(\obg{{k_2}},\ldots,(\obg{{k_{\ell-2}}},(\obg{j},\obg{i}))\ldots))$ can be excluded from the generating set. We choose a path from $i$ to~$j$, that is, choose vertices $i_1,\ldots,i_q$ from $k_1,\ldots,k_{\ell-2}$ with the property that $\sK$ contains the edges $\{i,i_1\}$, $\{i_{1},i_{2}\}$, $\ldots,$ $\{i_{q-1},i_{q}\}$, $\{i_q,j\}$. We proceed by induction on the length of the path.  Induction starts from the commutator
$(\obg{j},\obg{i})=1$ corresponding to a one-edge path $\{i,j\}\in\sK$. Now assume that the path consists of $q+1$ edges.
Using the relation~\eqref{swap} we can move the elements $\obg{{i_1}},\obg{{i_2}},\ldots,
\obg{{i_q}}$ in $(\obg{{k_1}},(\obg{{k_2}},\cdots(\obg{{k_{\ell-2}}},(\obg{j},\obg{i}))\cdots))$ to the right and restrict ourselves to the commutator
$(\obg{{i_1}},(\obg{{i_2}},\cdots(\obg{{i_q}},(\obg{j},\obg{i}))\cdots))$. Using~\eqref{swap} together with the commutation relation $(\obg{{i_q}},\obg{j})=1$ coming from the edge $\{i_q,j\}\in\sK$ we convert the commutator
$(\obg{{i_1}},(\obg{{i_2}},\cdots(\obg{{i_q}},(\obg{j},\obg{i}))\cdots))$
to the commutator
$(\obg{j},(\obg{{i_1}},\cdots(\obg{{i_{q-1}}},(\obg{{i_q}},\obg{i}))\cdots))$ modulo commutators of shorter length. The latter contains the commutator
$(\obg{{i_1}},\cdots(\obg{{i_{q-1}}},(\obg{{i_q}},\obg{i}))\cdots)$ corresponding to a shorter path $\{i,i_1,\ldots,i_q\}$. By the inductive hypothesis it can be expressed through commutators of shorter length and therefore excluded from the set of generators.

We therefore obtain a generator set for $\raag_\sK'$ consisting of nested commutators
$(\obg{{k_1}},\cdots(\obg{{k_{\ell-2}}},(\obg{j},\obg{i}))\cdots)$ with $j$ and $i$ in different connected components of $\sK_{\{k_1,\ldots,k_{\ell-2},j,i\}}$. Consider commutators
$(\obg{{k_1}},\cdots(\obg{{k_{\ell-2}}},(\obg{j},\obg{{i_1}}))\cdots)$
and
$(\obg{{k'_1}},\cdots(\obg{{k'_{\ell-2}}},(\obg{j},\obg{{i_2}}))\cdots)$ with the property
that
$\{k_1,\ldots,k_{\ell-2},j,i_1\}=\{k'_1,\ldots,k'_{\ell-2},j,i_2\}$
and $i_1,i_2$ lie in the same connected component of
$\sK_{\{k_1,\ldots,k_{\ell-2},j,i_1\}}$ which is different from the connected component containing~$j$. We claim that one of these commutators can be expressed through the other and commutators of shorter length. To see this, we argue as in the previous paragraph, i.\,e. we consider a path between $i_1$ and~$i_2$ in $\sK_{\{k_1,\ldots,k_{\ell-2},j,i_1\}}$  and then reduce it inductively to a one-edge path. This leaves us with the pair of commutators
$(\obg{{i_2}},(\obg{j},\obg{{i_1}}))$ and
$(\obg{{i_1}},(\obg{j},\obg{{i_2}}))$, where $\{i_1,i_2\}\in\sK$,
$\{i_1,j\}\notin\sK$, $\{i_2,j\}\notin\sK$. The claim then follows easily from the relation $(\obg{{i_1}},\obg{{i_2}})=1$.

Now, to enumerate independent commutators, we use the convention of writing
$(\obg{{k_1}},\cdots(\obg{{k_{\ell-2}}},(\obg{j},\obg{i}))\cdots)$ where $i$ is the smallest vertex in its connected component within $\sK_{\{k_1,\ldots,k_{\ell-2},j,i\}}$. This leaves us with precisely the set of commutators~\eqref{itercommu}.

It remains to show that the generating set~\eqref{itercommu} is minimal. To do this we use the ``exhausting'' procedure as in the final part of the proof of Theorem~\ref{basisraagfree}. Recall that
$\raag'_\sK=\pi_1(\lk)$, see Theorem~\ref{artfund}. Let $I^m(s)$ be a lattice cube with edges of length~$s$. Set $\lk^{(s)}=\lk\cap I^m(s)$. Then 
$\lk^{(s)}$ is a finite cell complex (a cubic complex) and 
$G^{(s)}_m=\pi_1(\lk^{(s)})$ is a finitely generated group. We can assume $\R^m=\bigcup_{s\in\mathbb N}I^m(s)$, so that $\lk=\bigcup_{s\in\mathbb N}\lk^{(s)}$. The infinitely generated group $\raag'_\sK=\pi_1(\lk)$ is the inductive limit of the groups~$G^{(s)}_m$. It is enough to prove that the commutators~\eqref{itercommu} corresponding to the loops inside~$I^m(s)$ give a minimal generating set for the group~$G^{(s)}_m$. The fact that these commutators generate~$G^{(s)}_m$ follows from the argument above; we only need to check the minimality.

A commutator of the form~\eqref{itercommu} corresponds to a loop in~$I^m(s)$ if and only if all exponents of the elements in the commutator satisfy $0<n_k\le s$. The number of commutators~\eqref{itercommu} with exponents in this range is given by
\begin{equation}\label{Pm}
  P_m^{(s)}=\sum_{J\subset[m]}(\mathop{\mathrm{cc}}(\sK_J)-1)s^{|J|},
\end{equation}
where $\mathop{\mathrm{cc}}(\sK_J)$ is the number of connected components of the complex~$\sK_J$, and $|J|$ denotes the cardinality of the set~$J$. Consider the first integer homology group $H_1(\lk^{(s)})$. Since $G^{(s)}_m=\pi_1(\lk^{(s)})$, we have $H_1(\lk^{(s)})=G^{(s)}_m/(G^{(s)}_m)'$. On the other hand, the space $\lk^{(s)}$ is a polyhedral product of the form $(I_s,Z_{s+1})^\sK$, where $I_s=[0,s]$ is the segment of length~$s$ and
$Z_{s+1} = \{0, \ldots, s\}$ is the set integer point on this segment. Then Theorem~\ref{wedgethmX} implies that
\[
  H_1(\lk^{(s)})\cong
  \widetilde H_0\Bigl(\bigvee_{J\subset[m]}|\sK_J| \wedge Z_{s+1}^{\wedge J}\Bigr)
  \cong \widetilde H_0\Bigl(\bigvee_{J\subset[m]} |\sK_J|^{\vee
  (s^{|J|})}\Bigr).
\]
Comparing this with~\eqref{Pm} we obtain that $H_1(\lk^{(s)})$ is a free abelian group of rank~$P_m^{(s)}$ (because $\widetilde H_0(X\vee Y)=\widetilde H_0(X)\oplus\widetilde H_0(Y)$
and $\rank\widetilde H_0(|\sK_J|)=\mathop\mathrm{cc}(\sK_J)-1$).
Therefore, the number of generators of the group
$G^{(s)}_m=\pi_1(\lk^{(s)})$ is at least $P_m^{(s)}$, as needed.
\end{proof}

According to a result of of Servatius,
Droms and Servatius~\cite{s-d-s89}, the commutator subgroup $\raag'_\sK$ is free if and only if the $1$-skeleton $\sK^1$ is a chordal graph (see also~\cite[Corollary~4.4]{pa-ve16}). A graph without loops and double edges is called \emph{chordal} (or \emph{triangulated})
if each of its cycles with $\ge 4$ vertices has a chord (an edge
joining two vertices that are not adjacent in the cycle). When $\sK^1$ is a chordal graph, Theorem~\ref{raagbasis} gives a basis for the free group~$\raag'_\sK$.

\begin{example}\

\textbf{1.} Let $\sK$ be a $5$-cycle, shown in Figure~\ref{pentagon}, left. 
\begin{figure}[h]
\unitlength=0.8mm
  \begin{center}
  \begin{picture}(110,40)
  \put(8,-3){\small 5}
  \put(-2.5,18){\small 1}
  \put(19,40.5){\small 2}
  \put(41,18){\small 3}
  \put(30,-3){\small 4}
  \put(10,0){\circle*{1}}
  \put(30,0){\circle*{1}}
  \put(40,20){\circle*{1}}
  \put(20,40){\circle*{1}}
  \put(0,20){\circle*{1}}
  \put(10,0){\line(1,0){20}}
  \put(30,0){\line(1,2){10}}
  \put(40,20){\line(-1,1){20}}
  \put(20,40){\line(-1,-1){20}}
  \put(0,20){\line(1,-2){10}}
  \put(78,-3){\small 5}
  \put(67.5,18){\small 1}
  \put(89,40.5){\small 2}
  \put(111,18){\small 3}
  \put(100,-3){\small 4}
  \put(80,0){\circle*{1}}
  \put(100,0){\circle*{1}}
  \put(110,20){\circle*{1}}
  \put(90,40){\circle*{1}}
  \put(70,20){\circle*{1}}
  \put(80,0){\line(1,0){20}}
  \put(100,0){\line(1,2){10}}
  \put(110,20){\line(-1,1){20}}
  \put(90,40){\line(-1,-1){20}}
  \put(70,20){\line(1,-2){10}}
  \put(80,0){\line(1,4){10}}
  \put(100,0){\line(-1,4){10}}
  \end{picture}
  \end{center}
  \caption{A $5$-cycle (left) and a chordal graph (right).}
  \label{pentagon}
\end{figure}
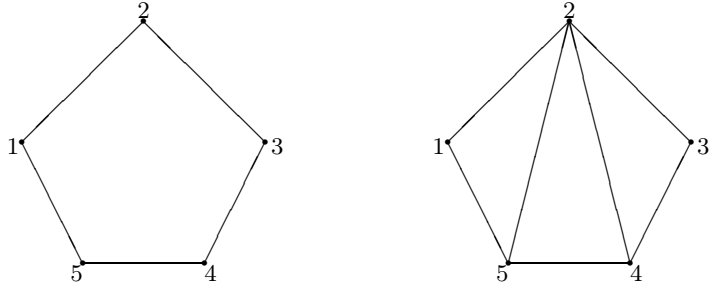
For the corresponding right-angled Artin group $\raag_\sK$, Theorem~\ref{raagbasis} gives a generator set for the commutator subgroup $\raag'_\sK$ consisting of the commutators
\begin{gather*}
  (g_3^{n_3},g_1^{n_1}),\;(g_4^{n_4},g_1^{n_1}),\;(g_4^{n_4},g_2^{n_2}),\;
  (g_5^{n_5},g_2^{n_2}),\;(g_5^{n_5},g_3^{n_3}),\\
  (g_4^{n_4},(g_5^{n_5},g_2^{n_2})),\;(g_3^{n_3},(g_5^{n_5},g_2^{n_2})),\;
  (g_1^{n_1},(g_5^{n_5},g_3^{n_3}),\;(g_3^{n_3},(g_4^{n_4},g_1^{n_1}),\;
  (g_2^{n_2},(g_4^{n_4},g_1^{n_1})
\end{gather*}
with $n_k \in \mathbb{Z}\setminus\{0\}$. The graph $\sK^1$ is not chordal, so the group $\raag'_\sK$
is not free. In fact, $\raag'_\sK$ contains a surface group of genus~$5$, generated by the ten commutators above with all exponents~$n_i=1$, which satisfy a single relation. This surface group is the commutator subgroup for the right-angled Coxeter group corresponding to~$\sK$, see~\cite[Example~4.8]{pa-ve16}.

\textbf{2.} Now let $\sK$ be the $1$-dimensional simplicial complex (graph) shown in~Figure~\ref{pentagon}, right. 
This graph is chordal, so the corresponding group $\raag'_\sK$ is free. Theorem~\ref{raagbasis} gives a basis consisting of the commutators
\begin{gather*}
  (g_3^{n_3},g_1^{n_1}),\;(g_4^{n_4},g_1^{n_1}),\;(g_5^{n_5},g_3^{n_3}),\;
  (g_1^{n_1},(g_5^{n_5},g_3^{n_3}),\;(g_3^{n_3},(g_4^{n_4},g_1^{n_1})
\end{gather*}
with $n_k \in \mathbb{Z}\setminus\{0\}$.
\end{example}

\section{Generalisation to graph products}\label{secgp}

\begin{definition}
Let $\sK$ be a simplicial complex on~$[m]$ and let $\mb
G=(G_1,\ldots,G_m)$ be a sequence of $m$ groups. The graph product $\mb G^{\sK}$ is defined as the quotient group
\[
  \mb G^{\sK}\cong\mathop{\mbox{\Huge$\star$}}_{k=1}^m G_k\big/(g_ig_j=g_jg_i\,\text{ for
}g_i\in G_i,\,g_j\in G_j,\,\{i,j\}\in\sK),
\]
where $\mathop{\mbox{\Huge$\star$}}_{k=1}^m G_k=G_1\star\cdots\star G_m$ denotes the free
product of the groups~$G_k$.
\end{definition}

The graph product $\mb G^{\sK}$ depends only on the $1$-skeleton of~$\sK$.
There is also a categorical definition of the graph product $\mb G^{\sK}$ as a colimit of the groups $\mb G^I=\prod_{i\in I}G_i$ over the category of faces $I\in\sK$, see~\cite[Construction~2.5]{pa-ve16} for the details. When each $G_i$ is $\Z$, the graph product $\mb G^{\sK}$ is the right-angled Artin group~$\raag_{\sK}$.

\begin{remark}
Although graph products can be defined for (well-behaved) topological groups~\cite{p-r-v04}, we only consider discrete groups $G_i$ here.
\end{remark}

There is a canonical epimorphism $\mb G^\sK\to\prod_{k=1}^m G_k$ obtained by letting $g_i\in G_i$ and $g_j\in G_j$ commute for any pair $\{i,j\}$.  When each group $G_k$ is abelian, the group
$\Ker(\mb G^\sK\to\prod_{k=1}^m G_k)$ coincides with the commutator subgroup $(\mb G^\sK){\vphantom{\bigr)}}'$.

\begin{theorem}\label{ragws}
Let $\sK$ be a flag simplicial complex on~$m$ vertices, let $\mb
G=(G_1,\ldots,G_m)$ be a sequence of $m$ nontrivial groups, and
let $\mb G^{\sK}$ be the corresponding graph product group.
\begin{itemize}
\item[(a)] $\Ker(\mb G^\sK\to\prod_{k=1}^m G_k)$ is a free group if and only if $\sK^1$ is a chordal graph;

\item[(b)] $\Ker(\mb G^\sK\to\prod_{k=1}^m G_k)$ has a minimal generator set consisting of iterated commutators
\[
  (g_j, g_i),\;
  (g_{k_1},(g_{j},g_{i})),\;\ldots,\;
  (g_{k_1},(g_{k_2},\ldots,(g_{k_{m-2}},(g_{j},g_{i}))\ldots)),
\]
where $g_k\in G_k\setminus\{1\}$, $k_1 < k_2 < \cdots < k_{m-2} < j > i$, $k_s \neq i$ for any $s$ and $i$ is the smallest vertex in a connected component not containing~$j$ of the subcomplex
$\sK_{\{k_1,\ldots,k_{\ell-2},j,i\}}$.
\end{itemize}
\end{theorem}

Part (a) is proved in~\cite[Theorem~4.3]{pa-ve16}.  This result may look surprising, as the condition in the criterion depends only on the combinatorics of $\sK$ and does not depend on the group structure of particular $G_k$. In particular, when $\sK$ is $m$ disjoint points, we obtain that
\[
  K_m:=\Ker\bigl(G_1\star\cdots\star G_m\to G_1\times\cdots\times G_m\bigr)
\]
is a free group. The reason is that $\Ker(\mb G^\sK\to\prod_{k=1}^m G_k)$ is the fundamental group of the polyhedral product $(E\mb G,\mb G)^\sK$, and therefore it depends only on the cardinality of $G_i$ (as the groups are discrete), and does not depend on their group structure.

The proof of~(b) follows the same line as for Theorem~\ref{raagbasis}. We give a sketch emphasising the most crucial points; a more detailed proof is included in the PhD thesis of the second author. 

\begin{proof}[Proof of Theorem~\ref{ragws}~(b)]
First, we consider the group $K_m$ corresponding to the case when $\sK$ is $m$ disjoint points. Observe that $K_2=\Ker(G_1\star G_2\to G_1\times G_2)$ is generated by the commutators $(g_2,g_1)$ with $g_i\in G_i\setminus\{1\}$. Indeed, a word $g_{i_1}g_{j_1}g_{i_2}\cdots g_{i_p}g_{j_p}\in G_1\star G_2$ with $g_{i_k}\in G_1$, $g_{j_k}\in G_2$, $k=1,\ldots,m$, belongs to $K_2$ if and only if $g_{i_1}\cdots g_{i_p}=1$ in $G_1$ and $g_{j_1}\cdots g_{j_p}=1$ in~$G_2$. Now we split off commutators from the word in the same way as in the case $m=2$ of Theorem~\ref{basisraagfree}. In the case of $K_m$ we argue by induction as follows. Write an element of $K_m$ as $g_{i_1}h_{j_1}g_{i_2}\cdots g_{i_p}h_{j_p}$ with $g_{i_k}\in G_m$, $h_{j_k}\in G_1\star\cdots\star G_{m-1}$ and $g_{i_1}\cdots g_{i_p}=1$ in $G_m$. Then split off commutators from the left as 
\begin{multline*}
  g_{i_1}h_{j_1}g_{i_2}\cdots g_{i_p}h_{j_p}=(g_{i_1}^{-1},h_{j_1}^{-1})h_{j_1}     
  g_{i_1}g_{i_2}\cdots g_{i_p}h_{j_p}\\=(g_{i_1}^{-1},h_{j_1}^{-1})
  (h_{j_1}^{-1},(g_{i_1}g_{i_2})^{-1})\cdots g_{i_p}h_{j_p}=\ldots
\end{multline*}
until we end up at a product of the form $c_1\cdots c_q h'$ where each $c_i$ is a commutator of the form $(g,h)$ or $(h,g)$ with $g\in G_m$, $h\in G_1\star\cdots\star G_{m-1}$, and $h'\in K_{m-1}$. Then we write $h'$ in a similar way using the inductive assumption. Now we use identities like $(ab,c)=(a,c)((a,c),b)(b,c)$ to show that $K_m$ is generated by iterated commutators $(g_1,\ldots,g_p)$ of length $p\ge2$ with arbitrary bracketing and $g_i\in G_{\alpha_i}$.

The next step is to express an iterated commutator $(g_1,\ldots,g_p)$ with arbitrary bracketing in terms of canonical nested commutators
$(g_1,\cdots(g_{p-2},(g_{p-1},g_p))\cdots)$; this is done is the standard way using the commutator identities as in~\cite{wald70} or~\cite{m-k-s76}. Then we use the commutator identity similar to~\eqref{swap}:
\[
  (g_q,(g_p,x))=(g_{q},x)(x,(g_{p},g_{q}))(g_{q},g_{p})(x,g_{p})(g_{p},(g_{q},x))
  (x,g_{q})(g_{p},g_{q})(g_{p},x)
\]
(here $g_p$ is an arbitrary element of $G_{\alpha_p}$, while in~\eqref{swap} it denoted a generator of $G_{\alpha_p}\cong\Z$). It allows us to swap elements in a canonical nested commutator until they reach the positions prescribed by the ordering of indices in~\eqref{itercomm}. This shows that the group $K_m$ is generated by precisely the commutators described in Theorem~\ref{ragws}~(b).

To show that the described commutators generate the group  $G(\sK):=\Ker(\mb G^\sK\to\prod_{k=1}^m G_k)$ for arbitrary $\sK$ we argue exactly as in the proof of Theorem~\ref{raagbasis}: eliminate commutators not appearing in the generator set using the new commutation relations in $\mb G^\sK$ and the commutator identity from the previous paragraph.

When establishing the minimality of the generator set, we replace $\lk=(\R,\Z)^\sK$ by the polyhedral product $(E\mb G,\mb G)^\sK$, which is the classifying space for the group
$G(\sK)$ by~\cite[Theorem~1.1]{staf15} or~\cite[Theorem~3.2]{pa-ve16}. First assume that each $G_k$ is a finite group. Then the number of commutators in Theorem~\ref{ragws}~(b) is finite and given by
\[
  P_m=\sum_{J\subset[m]}\Bigl((\mathop{\mathrm{cc}}(\sK_J)-1)\prod_{j\in J}(|G_j|-1)\Bigr),
\]
where $\mathop{\mathrm{cc}}(\sK_J)$ is the number of connected components of the complex~$\sK_J$, and $|G_j|$ denotes the cardinality of~$G_j$.
On the other hand, Theorem~\ref{wedgethmX} implies that
\begin{multline*}
  G(\sK)/G'(\sK)=H_1\bigl((E\mb G,\mb G)^\sK\bigr)\\
  \cong
  \widetilde H_0\Bigl(\bigvee_{J\subset[m]}|\sK_J| \wedge \mb G^{\wedge J}\Bigr)
  \cong \widetilde H_0\Bigl(\bigvee_{J\subset[m]} 
  |\sK_J|^{\vee \prod_{j\in J}(|G_j|-1)}\Bigr).
\end{multline*}
Hence, $G(\sK)/G'(\sK)$ is a free abelian group of rank~$P_m$.
Therefore, the number of generators of the group
$G(\sK)$ is at least $P_m$, as needed.

In the case when each $G_k$ is countable infinite we can use the ``exhausting'' procedure as in Theorem~\ref{raagbasis}. Note that the group-theoretic properties of $G_k$ do not affect this argument; only the cardinalities matter.

In the case of arbitrary (discrete) $G_k$, if the generator set is not minimal, then some iterated commutator $(g_1,\ldots,g_p)$ from the given generator set can be eliminated by expressing it through other commutators in the generator set. Denote by $H_k\subset G_k$ the subgroup generated by all elements of $G_k$ entering this expression. Then the same expression gives a relation in the graph product group $\mb H^{\sK}$, showing that the generator set is also not minimal for~$\mb H^{\sK}$. However, each group in the sequence $\mb H=(H_1,\ldots,H_m)$ is finitely generated and therefore countable, which is a contradiction.
\end{proof}

\end{document}